\documentclass[11pt]{amsart}
\usepackage{amsfonts, amstext, amsmath, amsthm, amscd, amssymb, pb-diagram}

\theoremstyle{plain}
\newtheorem{theorem}{Theorem}[section]
\newtheorem{corollary}[theorem]{Corollary}
\newtheorem{lemma}[theorem]{Lemma}
\newtheorem{prop}[theorem]{Proposition}

\newtheorem*{no-num-theorem}{Theorem}

\theoremstyle{definition}
\newtheorem{define}[theorem]{Definition}
\newtheorem*{remark}{Remark}

\def \R {\mathbf{R}}
\def \Z {\mathbf{Z}}
\def \C {\mathbf{C}}

\def \S {\mathbf{S}}

\def\RR{\mathcal{R}}
\def\DD{\mathcal{D}}

\numberwithin{equation}{section}

\def\SLC{{\rm SL}(2,\C)}
\def\SU{{\rm SU}(2,1)}
\def\PU{{\rm PU}(2,1)}
\def\SUtwo{{\rm SU}(2)}
\def\SUn{{\rm SU}(n)}

\def\su{{\frak {su}(2,1)}}
\def\U21{{\rm U}(2,1)}
\def\u21{{\frak {u}(2,1)}}
\def\Utwo{{\rm U}(2)}
\def\Uone{{\rm U}(1)}

\def\fg{\pi_1}

\begin{document}
\title[On the Burns-Epstein invariants of spherical CR 3-manifolds]{On the Burns-Epstein invariants of  \\ spherical CR 3-manifolds}
\author {Vu The Khoi}
\address{ Institute of Mathematics, 18 Hoang Quoc Viet road, 10307, Hanoi, Vietnam}
\email{vtkhoi@math.ac.vn}

\begin{abstract}
In this paper we develop a method to compute the Burns-Epstein invariant of a spherical CR homology sphere, up to an integer, from its holonomy representation. As an application, we give a formula for the Burns-Epstein invariant, modulo an integer, of a spherical CR structure on a Seifert fibered homology sphere  in terms of its holonomy representation. 
\end{abstract}
\maketitle
\section{Introduction} 
\vskip0.1cm
In \cite{be1}, Burns and Epstein define a global, biholomorphic, $\R$-valued invariant $\mu$ of a compact, strictly 
pseudoconvex $3$-dimensional CR manifold $M$ whose holomorphic tangent bundle is trivial. 
As the Burns-Epstein invariant is defined through the transgression form, it depends on the Cartan connection of the CR 
structure in a delicate way. Therefore it is not easy to compute the Burns-Epstein invariants for a general CR 3-manifold. In Burns-Epstein's work \cite{be1}, they compute the Burns-Epstein invariant for tangent circle bundles over Riemann surfaces and  Reinhardt domain in $\C^2.$ In \cite{be2}, Burns and Epstein raise the following question: ''An interesting question is left open here about the relationship of these invariants to the K$\ddot {\text a}$hler geometry of the interior manifold, and the behavior of developing maps for CR manifolds which are locally CR equivalent to the standard sphere''. 

This paper is an attempt to answer the second part of this question.
Namely, we show that for a spherical CR homology sphere, the Burns-Epstein invariant, modulo an integer, is basically a 
"topological'' invariant. More precisely, it coincides with minus the Chern-Simons 
invariant of the holonomy representation. The main result of our is the development of a cut-and-paste method, inspired from the works of P. Kirk and E. Klassen in gauge theory (\cite{kk1,kk2}), to compute the Burns-Epstein invariant, modulo an integer, of spherical CR homology spheres. We first defined the normal forms of a flat connection near the torus boundary and then prove a result (Theorem 5.1) which expresses the change of the Chern-Simons invariants of a path of normal form flat connections on a manifold with boundary in terms of the boundary holonomy. To compute the Chern-Simons invariant on a closed manifold $M,$ we decompose it as union of manifolds with torus boundary. On each manifold with boundary, we try to connect our original connection to a connection whose Chern-Simons invariant is already known and then use Theorem 5.1 to compute its Chern-Simons invariant.  
This method has been applied successfully to compute the Chern-Simons invariants of representations into $\SUtwo, \SLC, \SUn$ (\cite{kk1,kk2,nishi}) and the Godbillon-Vey invariants of foliations (\cite{k1}).  
 
The rest of this paper is organized as follows. In the next section, we recall some preliminaries about the universal covering group $G$ of  $\U21$ and the Burns-Epstein invariant of a CR manifold. In section 3, we show that, up to an integer, the Burns-Epstein invariant of spherical CR homology sphere equals minus the Chern-Simons invariant of its holonomy representation. 
Section 4 contains technical results which allow us to define a normal form of a flat $G$ connection on a manifold with boundary. In section 5, We prove Theorem 5.1 which expresses the change of the Chern-Simons invariant of a path of flat connections in terms of the boundary holonomy. This theorem are our tool to compute the Chern-Simons invariant in section 6, where an explicit formula for the Chern-Simons invariant of a Seifert fibered homology sphere is given. As an illustration, we carry out an explicit computation of the Burns-Epstein invariants, modulo an integer, of the homology sphere $\Sigma(2,3,11).$     

\noindent
{\bf Acknowledgment.}  
The main part of this paper was written during  2005-2006, when the author is supported by a COE Postdoctoral Fellowship from  the Graduate School of Mathematical Sciences, University of Tokyo and the  JSPS's Kakenhi Grant. The author would like to express his gratitude to the Graduate School of Mathematical Sciences for its support and especially to Professor Takashi Tsuboi his guidance and encouragement during that time.

\section{ Preliminaries}    
Recall that the group $\U21$ is the matrix group consists of all the $3\times 3$ matrices $A=(a_{ij})_{i,j=1,\dots,3}$ of complex entries such that $JA^\dagger J=A^{-1},$ where  $J$ is $diag(1,1,-1)$-the diagonal matrix whose main diagonal  are $(1,1,-1)$  and $A^\dagger $ is the complex conjugate of $A.$ Note that $\U21$ acts on the open unit ball in $\C^2,$ a model for the complex hyperbolic space $H^2_{\C}$, by :
$$(z,w) \longmapsto (\frac{a_{11}z+a_{12}w+a_{13}}{a_{31}z+a_{32}w+a_{33}},\frac{a_{21}z+a_{22}w+a_{23}}{a_{31}z+a_{32}w+a_{33}} ). $$
This action is transitive and the stabilizer of the origin is isomorphic to $\Utwo \times \Uone.$ It follows that the fundamental 
group of $\U21$ is isomorphic to $\Z\oplus \Z.$ 

In the following, we will construct the universal covering group
of $\U21,$ which is denoted by $G$ for short. Let's define 

\noindent
$G:=\{(A,\theta_1,\theta_2)\in \U21\times \R\times \R\ | \quad \theta_1\equiv \arg(\det(A))\mod 2\pi, \theta_2\equiv \arg(a_{33})\mod 2\pi \}.$

Since $A\in \U21,$ we find that  $|a_{31}|^2+|a_{32}|^2-|a_{33}|^2=-1.$  Therefore $a_{33}\ne 0,$ and the definition makes sense.

The multiplication on $G$ is defined by :
$$(A,\theta_1,\theta_2)(B, \phi_1, \phi_2)= (AB, \theta_1+\phi_1, \theta_2+\phi_2+ \arg(1+\frac{a_{31}b_{13}+a_{32}b_{23}}{a_{33}b_{33}})).$$  
Here, $\arg$ is the principal argument which takes value in $(-\frac{\pi}{2},\frac{\pi}{2}).$

Using the same argument as in \cite{hansen}, we can see that the multiplication is well-defined and $G$ is indeed a covering 
group of $\U21.$ To check that $G$ is simply connected, we consider its action on the open unit ball in $\C^2$ through the
 action of   the first component which  lies in $\U21.$
It is not hard to see that the action is transitive and the stabilizer of the origin is homeomorphic to 
$\SUtwo\times\R\times\R$. This implies that homotopically, $G$ is the same as $\SUtwo$ therefore it is simply connected.
We will also identify the Lie algebra of $G$ with $\u21-$ the Lie algebra of $\U21.$

  The elements of $\U21$ can be divided into 3 types according to their action on the complex hyperbolic space 
$H^2_{\C}$(see \cite{chen}). Namely, a matrix is called {\it elliptic} if it has a fixed point in $H^2_{\C}.$ We call it {\it 
parabolic} if it has a unique fixed point in $\overline {H^2_{\C}}$ and this lies on $\partial H^2_{\C}.$ And finally, a matrix is
 called {\it loxodromic} if it has exactly two fixed points in $\overline {H^2_{\C}}$ which lie on $\partial H^2_{\C}.$  We also classify element of the universal covering group $G$ as elliptic, parabolic or loxodromic if 
its image under the  projection map $G\rightarrow \U21$  is of the corresponding types.

Let $M$ be a smooth, compact, oriented 3-manifold. A \textit{contact structure} on $M$ is a oriented 2-plane field $V=\textrm{ker}\ \alpha,$ where $\alpha$ is an 1-form  such that $\alpha\wedge d\alpha$ is nowhere zero.
A \textit{strictly pseudoconvex CR structure} on $M$ is a contact structure $V$ together with a complex structure $J$ on $V.$ Let $V\otimes\C=v\oplus\bar v,$ where $v,\bar v$ are the $i$ and $-i$ eigenspaces of $J$ respectively, then $v$ is called the \textit{holomorphic tangent bundle}. 

 A CR structure which is locally isomorphic to the standard CR structure on the unit sphere $\S^3\subset \C^2,$ is called a \textit{spherical CR structure.} A spherical CR structure is determined by a pair $(D, \rho),$ where $D: \tilde M\rightarrow \S^3$ is a local isomorphism and $\rho: \fg(M) \rightarrow \PU $ is the holonomy representation such that   
$ D\circ \gamma =\rho(\gamma) \circ D,\ \text{for all} \ \gamma\in \fg(M). $ See \cite{falbel} for more details about spherical CR structures.

We now briefly recalled the definition of the Burns-Epstein invariant. The reader is referred to \cite{be1} for more details.
Let (M,V,J) be a CR 3-manifold with trivial holomorphic tangent bundle and $\pi_Y:Y\rightarrow M$ be its CR structure bundle. Denoted by $\pi$ the Cartan connection form, that is a $\su$-valued 1-form on $Y.$ Let $\Pi=d\pi+\pi\wedge\pi$ be its curvature form. Consider a 3-form: $$TC_2(\pi):=\frac{1}{8\pi^2}Tr(\pi\wedge\Pi+\frac{1}{3}\pi\wedge\pi\wedge\pi).$$ The main theorem of Burns-Epstein \cite{be1} says that : there exists a 3-form $\tilde TC_2(\pi)$ on $M,$ which is defined up to an exact form, such that $\pi_Y^*(\tilde TC_2(\pi))= TC_2(\pi).$ Moreover, the integral $\int_M\tilde TC_2(\pi)$ is a biholomorphic invariant of the CR structure on $M.$ For a given CR 3-manifold $(M,V,J),$    this integral is simply denoted by $\mu(M)$ and called  the \textit{Burns-Epstein invariant} of $M.$

Since the Burns-Epstein invariant is only defined when the holomorphic tangent bundle is trivial, for simplicity we will restrict ourself to the case of homology spheres so that this condition is automatically fulfilled. With some modifications, the reader may extend our results here to the relative version of the Burns-Epstein on a general $3$-manifold as defined in \cite{cheng}.   

\section{Relation to  the Chern-Simons invariant}

Next, we will recall the definition of the Chern-Simons invariant of a flat connection associated to a 
representation. The reader is referred to \cite{cs} for general facts about the Chern-Simons invariant. For technical reason, we will work with the universal covering group $G.$ 
 
Let $\rho:\fg(M)\longrightarrow G$ be a representation of the fundamental group of $M.$ Consider the flat $G$ bundle 
associated to $\rho:\ E_{\rho}:=\tilde M\times_{\rho}G.$ Let $A$ be the connection form of the flat connection on
 $E_{\rho},$  then the Chern-Simons form of $A$ is defined by :
$$CS(A):=\frac{1}{8\pi^2}Tr(A\wedge dA+\frac{2}{3}A\wedge A\wedge A).$$

As we have shown that homotopically $G$ is the same as $\SUtwo,$ standard obstruction theory implies that $E_{\rho}$ is a trivial
bundle. The \textit{Chern-Simons invariant} of $\rho$ is defined by:
$$cs(\rho):=\int_Ms^*(CS(A))\\ \mod \Z,$$ 
where $s$ is a section of $E_{\rho}.$ It is not hard to see that the Chern-Simons invariant is well-defined, since 
the difference between two sections is, homologically, some multiples of the fiber and the
 Chern-Simons form when restricted to the fiber is the generator of $H^3(G;\Z)\cong\Z.$ 

We can also defined the Chern-Simons invariant for a representation $\rho$ into the group $\U21, \SU$ or $\PU$ in the same manner as long as the associated bundle $E_{\rho}$ is trivial. However we will get nothing new, since in this case we can lift $\rho$ to a representation $\tilde\rho$
into $G$ and $cs(\rho)\equiv cs(\tilde \rho).$ To show this equality, note that the connection form $A$
 is induced from the pullback of the Maure-Cartan form by the projection map $\tilde M\times G\rightarrow G$ and moreover the 
Maure-Cartan forms preserve under pullback of the covering maps of Lie groups.

The next observation shows that in the case of a spherical CR structure the Burns-Epstein invariant, modulo an integer,  is minus 
the Chern-Simons invariant of the holonomy representation. So in this case the Burns-Epstein invariant, modulo  an integer, only depends on
 the holonomy representation.
\begin{prop}
Let $(M,V,J)$ be a spherical CR homology sphere whose holonomy representation is $\rho:\fg(M)\longrightarrow \PU$ then
$$\mu(M)\equiv -cs(\rho)\ \ \mod \Z.$$
\end{prop}         
\begin{proof}
Let $D:\tilde M\longrightarrow S^3$  be the developing map for the CR structure on $M.$ Denoted  
by $Y_{\tilde M}$ and $Y_{\S^3}$ the CR structure bundles on $\tilde M$ and $\S^3$ respectively.
 Consider the commutative diagram below
$$
\begin{diagram}
\node{Y_{\tilde M}}\arrow{e,t}{\DD}\arrow{s,l}{\pi_{\tilde M}}
\node{Y_{\S^3}}\arrow{s,r}{\pi_{\S^3}}\\
\node{\tilde M}\arrow{e,t}{D}
\node{\S^3}
\end{diagram}
$$
Let $A$ be the Cartan connection form on $Y_{\S^3},$ then $\DD^*(A)$ is the Cartan connection form on 
$Y_{\tilde M}.$ As $M$ is a homology sphere, the CR structure bundle of $M$ is trivial, that is, there exist an equivariant 
section $s:\tilde M\longrightarrow Y_{\tilde M}.$  Then the equivariant 3-form $s^*(TC_2(\DD^*(A)))$ on $\tilde M$ will 
induce a 3-form on $M$ which we denote by the same notation. By using the vanishing of the curvature, we see that 
$$s^*(TC_2(\DD^*(A)))=s^*(-CS(\DD^*(A))) =-CS((\DD\circ s)^*(A)).$$

On the other hand, the CR structure bundle on the standard sphere can be identified with the Lie group $\SU$ (see \cite{j}) and 
the Cartan connection $A$ is the Maure-Cartan form on $\SU.$ Under this identification, the equivariant map
$\DD\circ s$ can be consider as a section of the bundle $E_{\rho}=\tilde M\times_{\rho}\SU.$ It now follows from the 
definition of the Burns-Epstein invariant that:
$$\mu(M)\equiv \int_M s^*(TC_2(\DD^*(A)))\equiv \int_M -CS((\DD\circ s)^*(A)) \equiv -cs(\rho) \mod \Z.$$
\end{proof}
According to \cite{j}, over the standard $3$-sphere the form $\tilde TC_2(A)$ coincides with  $-\frac{1}{2\pi^2}dVol,$ where 
$dVol$ is the volume form on the unit $3$-sphere. So it
 follows from the proof above that:
$$\mu(M)= \int_M (\DD\circ s)^*(TC_2(A))=\int_M (\DD\circ s)^*(\pi_{\S^3})^*(-\frac{1}{2\pi^2}dVol)$$
$$ =-\frac{1}{2\pi^2}\int_M D^*(dVol).\hskip4.6cm$$
 If $\kappa: \tilde M \rightarrow \S^3$ is an equivariant map, considered as a section of the associated bundle $\tilde M\times_\rho \S^3,$ then 
$\mu(M)\equiv -\int_M \kappa^*(\frac{1}{2\pi^2}dVol) \  \mod \Z.$ The reason is that two sections $D$ and $\kappa$ are homologically 
differed by some multiples of the fiber $\S^3$ and the integral of the form $\frac{1}{2\pi^2}dVol$ over the unit $3$-sphere is $1.$  We get the following corollary: 
 
\begin{corollary}Let $M$ be a spherical CR homology sphere whose holonomy representation 
$\rho:\fg(M)\longrightarrow \PU $ is reducible then $\mu(M)\equiv 0 \mod \Z.$ 
\end{corollary}
\begin{proof}As the holonomy representation $\rho$ is reducible, we can find a common fixed point $*\in \S^3$ for all elements in its image.  So we can take the constant map as a section of the associated bundle $\tilde M\times_\rho \S^3$ and the corollary follows. 
\end{proof}

\section{Normal forms of flat connections near the boundary torus}

On a manifold with boundary, the integral of the Chern-Simons form may depend on the way we choose the section $s.$ In order to define the Chern-Simons invariant on the manifold with boundary  we will firstly define explicit normal forms of flat connections near the boundary. We then show that every flat connection can be gauge transformed into a normal form. This will be done by finding explicit form, near the boundary, of the developing map associated to the flat connection. 

Let $X$ be a 3-manifold whose boundary $\partial X$ is a torus $T.$ We fixed a pair of meridian and longitude $\mu,\lambda$ on
 $T.$ Choose a coordinate $(e^{2\pi ix},e^{2\pi iy})$ on $T$ such that the corresponding map:
\begin{eqnarray*} 
\R^2 & \longrightarrow & T \\
(x,y) &\longmapsto &(e^{2\pi ix},e^{2\pi iy})
\end{eqnarray*}
sends the horizontal line to $\mu$ and the vertical  line $\lambda.$  
Suppose that $T\times [0,1]$ is the
collar neighborhood of $T$ in $X$ and $\{dx,dy,dr\} $ is an oriented basis of 1-forms
on $X$ near $T.$ Here $r$ is the coordinate on $[0,1]$ such that $T\times \{1\}=\partial X$ and we
orient $T$ by the ''outward normal last" convention.

Let $A$ be a flat connection form on a principal $G$-bundle over $X$ with holonomy $\rho.$  As the bundle is trivial, from now 
on we will consider $A$ as an $\u21$-valued $1$-form on $X.$ Recall that the developing map 
of $A$ is a map $D_A: \tilde X \longrightarrow G$ such that $D_A(\alpha \circ \tilde x)=\rho(\alpha)\circ D_A(\tilde x)$ for 
all $\alpha \in \fg(X)$ and  $\tilde x \in \tilde X.$ 
   
 We will follow the scheme in \cite{k1} and define the normal form for a flat $G$ connection on $X$ by dividing into 
several cases according to the type of the boundary holonomies.

\noindent (I) Elliptic: suppose that the boundary holonomies $\rho(\mu)$ and $\rho(\lambda)$ are elliptic, then by conjugation we may 
assume that: 

$\rho(\mu)=(A, 2\pi(\alpha_1+\alpha_2+\alpha_3), 2\pi\alpha_3)$ and $\rho(\lambda)=(B, 2\pi(\beta_1+\beta_2+\beta_3), 2\pi\beta_3).$ Where 

$A=\begin{pmatrix}
 e^{2\pi i\alpha_1}& 0& 0\\
0& e^{2\pi i\alpha_2}& 0 \\
0& 0 & e^{2\pi i\alpha_3}\end{pmatrix}, \ \ $ $B=\begin{pmatrix} 
 e^{2\pi i\beta_1}& 0& 0\\
0& e^{2\pi i\beta_2}& 0 \\
0& 0 & e^{2\pi i\beta_3}
\end{pmatrix}$ and
 $\alpha_i$ and $\beta_i$ are real numbers.  

\noindent 
We can see that near the boundary the developing map is given by: 
$ D: \R^2\times [0, 1]\longrightarrow G$ such that $D(x,y,r)=(M, 2\pi(\alpha_1x+\alpha_2x+\alpha_3x+\beta_1y+\beta_2y+\beta_3y), 2\pi(\alpha_3x+\beta_3y))$
where $$ M= \begin{pmatrix} 
 e^{2\pi i(\alpha_1x+\beta_1y)}& 0& 0\\
0& e^{2\pi i(\alpha_2x+\beta_2y)}& 0 \\
0& 0 & e^{2\pi i(\alpha_3x+\beta_3y)}
\end{pmatrix}.$$ 
It follows that near $\partial X=T,$ we can gauge transform the flat connection $A$ to the form:
$$ D^{-1}dD= \begin{pmatrix} 
 2\pi i(\alpha_1dx+\beta_1dy)& 0& 0\\
0& 2\pi i(\alpha_2dx+\beta_2dy)& 0 \\
0& 0 & 2\pi i(\alpha_3dx+ \beta_3dy)
\end{pmatrix}.$$
\noindent (II) Loxodromy: It follows from \cite{chen} Lemma 3.2.2 that in this case the boundary holonomy can be conjugated to the form 
$\rho(\mu)=(A, 2\pi\theta_1+4\pi\theta_2, 2\pi\theta_2)$ and $\rho(\lambda)=(B, 2\pi\tau_1+4\pi\tau_2, 2\pi\tau_2).$ Where 
$$A=\begin{pmatrix} 
 e^{2\pi i\theta_1}& 0& 0\\
0& e^{2\pi i\theta_2}\cosh u& e^{2\pi i\theta_2}\sinh u \\
0& e^{2\pi i\theta_2}\sinh u & e^{2\pi i\theta_2}\cosh u
\end{pmatrix},$$
$$B=\begin{pmatrix} 
 e^{2\pi i\tau_1}& 0& 0\\
0& e^{2\pi i\tau_2}\cosh v& e^{2\pi i\tau_2}\sinh v \\
0& e^{2\pi i\tau_2}\sinh v & e^{2\pi i\tau_2}\cosh v
\end{pmatrix}$$
and  $\theta_i, \tau_i, u, v$ are real numbers.
The developing map near the boundary is given by $$D(x,y,r)=(M, 2\pi(\theta_1x+\tau_1y) +4\pi(\theta_2x+\tau_2y),2\pi(\theta_2x+\tau_2y))$$ with 
$$ M= \begin{pmatrix} 
 e^{2\pi i(\theta_1x+\tau_1y)}& 0& 0\\
0& e^{2\pi i(\theta_2x+\tau_2y)}\cosh(ux+vy)& e^{2\pi i(\theta_2x+\tau_2y)}\sinh(ux+vy) \\
0& e^{2\pi i(\theta_2x+\tau_2y)}\sinh(ux+vy) & e^{2\pi i(\theta_2x+\tau_2y)}\cosh(ux+vy)
\end{pmatrix}.$$ 
So near the boundary the connection has the form
$$D^{-1}dD= \begin{pmatrix} 
2\pi i(\theta_1 dx+\tau_1 dy) & 0& 0\\
0& 2\pi i(\theta_2 dx+\tau_2 dy)& udx+vdy  \\
0& udx+vdy & 2\pi i(\theta_2 dx+\tau_2 dy)
\end{pmatrix} $$

\noindent (III) Parabolic: According to  \cite{chen} Theorem 3.4.1
 a parabolic element $g \in \U21 $ has a unique decomposition $g=pe=ep,$ where $p$ is unipotent and $e$ is elliptic.
 Furthermore, the classification of parabolic elements of $\U21$ can be divided into two cases corresponding to whether the minimal polynomial of the unipotent  part $p$ is $(x-1)^3$ or $(x-1)^2.$ 

If the minimal polynomial of  $p$ is $(x-1)^3$ then it can be conjugate to the form:
$ \begin{pmatrix} 
 1-s& \bar a & s\\
-a & 1& a \\
-s& \bar a & 1+s
\end{pmatrix}e^{2\pi i\alpha}  $
, where $\Re(s)= \frac{|a|^2}{2}.$ Note that by further conjugation by an appropriate diagonal matrix, we may assume 
that $a$ is real. If $g$ has the unipotent  part $p$ with the minimal polynomial  $(x-1)^2$ 
then it can be conjugate to the form: 
$$ \begin{pmatrix}
 e^{2\pi i\theta_1}(1-ip)& 0& ipe^{2\pi i\theta_1}\\
0& e^{2\pi i\theta_2}& 0 \\
-ipe^{2\pi i\theta_1}& 0 & e^{2\pi i\theta_1}(1+ip)\end{pmatrix}.$$
Therefore, in the following we will divide into two cases.

\noindent Case 1: By conjugation, we may assume that,  $$\rho(\mu)=(A, 6\pi\alpha, \arctan(\frac{p}{1+a^2/2})+2\pi\alpha),$$ 
$$\text{where}\ A=\begin{pmatrix} 
 1-\frac{a^2}{2}-ip& a & \frac{a^2}{2}+ip\\
-a & 1& a \\
-\frac{a^2}{2}-ip& a & 1+\frac{a^2}{2}+ip
\end{pmatrix}e^{2\pi i\alpha}.$$ 
As $\rho(\lambda)$ is of the same type and commutes with $\rho(\mu),$ it is not hard to check 
that   $\rho(\lambda)$ must have the form  
$$\rho(\lambda)=(B, 6\pi\beta, \arctan(\frac{q}{1+b^2/2})+2\pi\beta),$$ 
$$\text{where}\ B=\begin{pmatrix} 
 1-\frac{b^2}{2}-iq& b & \frac{b^2}{2}+iq\\
-b & 1& b \\
-\frac{b^2}{2}-iq& b & 1+\frac{b^2}{2}+iq
\end{pmatrix}e^{2\pi i\beta}.$$
 Note that  $a,b,p,q, \alpha, \beta$ are real numbers.
  
We then can choose the developing map to be 

$$D(x,y,r)=(M, 6\pi(\alpha x+\beta y), \arctan(\frac{px+qy}{1+(ax+by)^2/2})+2\pi(\alpha x+\beta y) ),$$
where 
 $$ M=\begin{pmatrix} 
 1-\frac{(ax+by)^2}{2}-& (ax+by)& \frac{(ax+by)^2}{2}+\\
i(px+qy)& &i(px+qy)\\ 
& &\\
-(ax+by)&1 & (ax+by) \\
& &\\
-\frac{(ax+by)^2}{2}&(ax+by)& 1+\frac{(ax+by)^2}{2}+\\
-i(px+qy)&& i(px+qy)
\end{pmatrix}e^{2\pi i(\alpha x+\beta y)}.$$
With straightforward computations, we deduce that the connection form $D^{-1}dD$ near the boundary is
$$ \begin{pmatrix} 
 -i(pdx+qdy)& (adx+bdy)& i(pdx+qdy)\\
2\pi i(\alpha dx +\beta dy)& &\\
 & &\\
-(adx+bdy)& 2\pi i(\alpha dx +\beta dy)& (adx+bdy) \\
& &\\
-i(pdx+qdy)& (adx+bdy) & i(pdx+qdy)+ \\
& &2\pi i(\alpha dx +\beta dy)
\end{pmatrix}.$$

\noindent Case 2: After conjugation, we may assume that  $$\rho(\mu)=(A, 4\pi\theta_1+2\pi\theta_2, 2\pi\theta_1+\arctan(p))$$   $$\text{and}\ \rho(\lambda)=(B, 4\pi\tau_1+2\pi\tau_2, 2\pi\tau_1+\arctan(q)).$$  

$$ \text{Where}\ A=\begin{pmatrix}
 e^{2\pi i\theta_1}(1-ip)& 0& ipe^{2\pi i\theta_1}\\
0& e^{2\pi i\theta_2}& 0 \\
-ipe^{2\pi i\theta_1}& 0 & e^{2\pi i\theta_1}(1+ip)\end{pmatrix}, $$

$$  B=\begin{pmatrix}
 e^{2\pi i\tau_1}(1-iq)& 0& iqe^{2\pi i\tau_1}\\
0& e^{2\pi i\tau_2}& 0 \\
-iqe^{2\pi i\tau_1}& 0 & e^{2\pi i\tau_1}(1+iq)\end{pmatrix},$$ and $\theta_i, \tau_i, p, q$ are real numbers. So the developing map is of the form 

$$D(x,y,r)=(M, 4\pi(\theta_1 x+\tau_1 y)+ 2\pi(\theta_2 x+\tau_2 y), 2\pi (\theta_1 x+\tau_1 y)+\arctan(px+qy)),$$
where
$$M= \begin{pmatrix}
 e^{2\pi i(\theta_1 x+\tau_1 y)}- & 0& i(px+qy)e^{2\pi i(\theta_1 x+\tau_1 y)}\\
  i(px+qy)e^{2\pi i(\theta_1 x+\tau_1 y)}&&\\
  &&\\
0& e^{2\pi i(\theta_2 x+\tau_2 y)}& 0 \\
&&\\
-i(px+qy)e^{2\pi i(\theta_1 x+\tau_1 y)}& 0 & e^{2\pi i(\theta_1 x+\tau_1 y)}+\\
&& i(px+qy)e^{2\pi i(\theta_1 x+\tau_1 y)} \end{pmatrix}
.$$
After some lengthy computations, we find the connection form $D^{-1}dD$ to be:
 $$ \begin{pmatrix}
  2\pi i(\theta_1 dx+\tau_1 dy)-  & 0&  i(pdx+qdy)\\
  i(pdx+qdy)&&\\
  &&\\
0& 2\pi i(\theta_2 dx+\tau_2 dy)& 0 \\
&&\\
- i(pdx+qdy)& 0 & 2\pi i(\theta_1 dx+\tau_1 dy)+ \\
&& i(pdx+qdy)
\end{pmatrix}.$$
\begin{define} We say that a flat $G$ connection $1$-form on a manifold with torus boundary $X$ is in normal form if, near the boundary, it has one of the forms as in the cases (I), (II) and (III) above.
\end{define}
Note that one connection  may have different normal forms as we can add  integers to the exponential parameters in the holonomy matrices.

For the computation of the Burns-Epstein invariant, we need to bring connections to the normal forms. The following proposition shows that we can always do so.
 \begin{prop} a) Let $A$ be a flat $G$ connection $1$-form on $X.$  Then there exists a gauge transformation which brings $A$ into a normal form.   
\label{nf}
\noindent
b) Let $A_t$ be a path of flat $G$ connection $1$-forms on $X.$ Then there exists a path of gauge transformations $g_t$ such that  $g_t\cdot A_t$ is in normal form for all $t$.
\end{prop}         
\begin{proof}
We have shown above that locally near the torus boundary, we can always gauge transform the flat connection into a normal form. By using standard obstruction theory argument as in \cite{kk2} Proposition 2.3 we can extend the gauge transformation to all of $X$ and prove the proposition. 
\end{proof}
The next lemma tells us that the integral of the Chern-Simons form of  a normal form flat connection on a manifold with boundary is gauge invariant.
\begin{lemma}  
Let $A$ and $B$ be two normal form flat connections on a manifold with torus boundary $X.$ Suppose that:

1. $A$ and $B$ are gauge equivalent.

2. $A$ and $B$ are in normal form and equal near the boundary.

Then $\int_XCS(A)\equiv \int_XCS(B) \mod \Z.$
\end{lemma}
\begin{proof} The proof is similar to the one of \cite{kk2} Theorem 2.4. So we will left it to the reader as an exercise.
\end{proof}
So we may define the \textit{Chern-Simons invariant of a normal form flat connection} $A$ by:
$$cs(A):=\int_X CS(A) \mod \Z.$$ 
\section{Variations of the Chern-Simons invariants}

In this section, we will prove the main technical tool for computing the Chern-Simons invariant which is the formula for the change of the Chern-Simons invariants of a path of normal form flat connections in terms of the boundary holonomy.
\begin{theorem} Let $A_t$ be a path of normal form flat $G$-connections on a manifold with torus boundary $X.$ 
Suppose that $\rho_t:\fg(X)\longrightarrow G$ is the corresponding path of holonomy representations. Consider the 
following cases:

\noindent (I)Elliptic: suppose that  $\rho_t(\mu)$ and $\rho_t(\lambda)$ are elliptic elements and the normal form of $A_t$ near the boundary  is:
$$\begin{pmatrix} 
 2\pi i(\alpha_1(t)dx+\beta_1(t)dy)& 0& 0\\
0& 2\pi i(\alpha_2(t)dx+\beta_2(t)dy)& 0 \\
0& 0 & 2\pi i(\alpha_3(t)dx+ \beta_3(t)dy)
\end{pmatrix}$$
then $$cs(A_1)-cs(A_0)\equiv \frac{1}{2}\int_0^1 [(\alpha_1\dot{\beta_1}-\dot{\alpha_1}\beta_1)+ (\alpha_2\dot{\beta_2}-\dot{\alpha_2}\beta_2)
 +( \alpha_3\dot{\beta_3}-\dot{\alpha_3}\beta_3)]dt. $$

\noindent (II) Loxodromy:  if $\rho_t(\mu)$ and $\rho_t(\lambda)$ are loxodromy  elements and the normal form of $A_t$ near the boundary  is:
$$\begin{pmatrix} 
2\pi i(\theta_1(t) dx+\tau_1(t) dy) & 0& 0\\
0& 2\pi i(\theta_2(t) dx+\tau_2(t) dy)& u(t)dx+v(t)dy  \\
0& u(t)dx+v(t)dy & 2\pi i(\theta_2(t) dx+\tau_2(t) dy)
\end{pmatrix} $$
then
$$cs(A_1)-cs(A_0)\equiv \frac{1}{2}\int_0^1 ( \theta_1\dot{\tau_1}-\dot{\theta_1}\tau_1)dt +\int_0^1(\theta_2\dot{\tau_2}-\dot{\theta_2}\tau_2)dt+
\frac{1}{4\pi^2}\int_0^1 (\dot{u}v- u\dot{v})dt.$$

\noindent (III) Parabolic:

\noindent
Case 1:  $\rho_t(\mu)$ and $\rho_t(\lambda)$ are parabolic elements whose unipotent parts have minimal polynomial $(x-1)^3.$ Suppose that the normal form of $A_t$ near the boundary  is: 
$$ \begin{pmatrix} 
 -i(p(t)dx+q(t)dy)+ & a(t)dx+b(t)dy& i(p(t)dx+q(t)dy)\\
 2\pi i(\alpha(t) dx +\beta(t) dy)&\ &\\ 
 &&\\
-(a(t)dx+b(t)dy)& 2\pi i(\alpha(t) dx +\beta(t) dy)& a(t)dx+b(t)dy \\
&&\\
-i(p(t)dx+q(t)dy)& a(t)dx+b(t)dy & i(p(t)dx+q(t)dy)+ \\
\ &\ & 2\pi i(\alpha(t) dx +\beta(t) dy)
\end{pmatrix}$$  then  
$$cs(A_1)-cs(A_0)\equiv \frac{3}{2}\int_0^1 ( \alpha\dot{\beta}-\dot{\alpha}\beta)dt.$$
\noindent
Case 2:  $\rho_t(\mu)$ and $\rho_t(\lambda)$ are parabolic elements whose unipotent parts have minimal polynomial $(x-1)^2.$ Suppose that the normal form of $A_t$ near the boundary  is: 
$$ \begin{pmatrix}
  2\pi i(\theta_1(t) dx+\tau_1(t) dy)-  & 0&  i(p(t)dx+q(t)dy)\\
  i(p(t)dx+q(t)dy)&\ &  \\
  &&\\
0& 2\pi i\theta_2(t) dx+& 0 \\
&2\pi i\tau_2(t) dy&\\
&&\\
- i(p(t)dx+q(t)dy)& 0 & 2\pi i(\theta_1(t) dx+\tau_1(t) dy)+ \\
&\ & i(p(t)dx+q(t)dy) \end{pmatrix}$$ 
then
$$cs(A_1)-cs(A_0)\equiv \int_0^1 (\theta_1\dot{\tau_1}-\dot{\theta_1}\tau_1)dt+ \frac{1}{2}\int_0^1(\theta_2\dot{\tau_2}-\dot{\theta_2}\tau_2)dt.$$
\end{theorem}
\begin{proof}
Consider the path $A_t$ as a connection on $X\times[0, 1],$  then it is a flat connection, that is $F^{A_t}\wedge F^{A_t}\equiv 0.$ by Stokes' Theorem
$$\int_XCS(A_1) - \int_XCS(A_0)-\int_{\partial X\times [0,1]} CS(A_t)=\int_{X\times [0,1]}tr(F^{A_t}\wedge F^{A_t})=0.$$  
So $$cs(A_1)-cs(A_0) \equiv \int_{\partial X\times [0,1]} \frac{1}{8\pi^2}tr(A_t\wedge dA_t+\frac{2}{3}A_t\wedge A_t\wedge A_t).$$
We now use this formula and the normal form of the connections to compute the changes of Chern-Simons invariants in each 
cases.

\noindent (I) Elliptic: in this case
$tr(A_t\wedge dA_t+\frac{2}{3}A_t\wedge A_t\wedge A_t)=4\pi^2[( \alpha_1\dot{\beta_1}-\dot{\alpha_1}\beta_1) + 
(\alpha_2\dot{\beta_2}-\dot{\alpha_2}\beta_2)  + (\alpha_3\dot{\beta_3}-\dot{\alpha_3}\beta_3) ]dx\wedge dy\wedge dt.$
 So the result follows.

\noindent (II) Loxodromy: after some computation, we get 

$tr(A_t\wedge dA_t+\frac{2}{3}A_t\wedge A_t\wedge A_t)=
4\pi^2( \theta_1\dot{\tau_1}-\dot{\theta_1}\tau_1)dx\wedge dy\wedge dt+$

$8\pi^2(\theta_2\dot{\tau_2}-\dot{\theta_2}\tau_2)dx\wedge dy\wedge dt+2(\dot{u}v- u\dot{v})dx\wedge dy\wedge dt.$

We deduce the required formula.

\noindent (III) Parabolic: 

\noindent Case 1: straightforward computations show that 
$$tr(A_t\wedge dA_t+\frac{2}{3}A_t\wedge A_t\wedge A_t)=12\pi^2(\alpha\dot{\beta}-\dot{\alpha}\beta )dx\wedge dy\wedge dt.$$
 So the change in the Chern-Simons invariants is given by the formula.

\noindent Case 2: in this case, we get: 

$tr(A_t\wedge dA_t+\frac{2}{3}A_t\wedge A_t\wedge A_t)=8\pi^2(\theta_1\dot{\tau_1}-\dot{\theta_1}\tau_1 )dx\wedge dy\wedge dt +$

$ 4\pi^2(\theta_2\dot{\tau_2}-\dot{\theta_2}\tau_2)dx\wedge dy\wedge dt.$

 So the formula follows.  
\end{proof}
In the last past of this section, we will study the difference between the Chern-Simons invariants of two different normal form connections in the elliptic case. This result will be used in the computation of the next section. Our result is similar to  Theorem 2.5 in \cite{kk2}.

Consider a manifold with toral boundary $X.$ Let A be a normal form flat $G$-connection which has the following form near the boundary $\partial X$:
$$\begin{pmatrix} 
 2\pi i(\alpha_1dx+\beta_1dy)& 0& 0\\
0& 2\pi i(\alpha_2dx+\beta_2dy)& 0 \\ 
0& 0 & 2\pi i(\alpha_3dx+ \beta_3dy)
\end{pmatrix}.$$

Let $h:\S^1\rightarrow G $ defined by $h(e^{2\pi i\theta})=(diag(e^{2\pi i\theta},e^{-2\pi i\theta},1),0,0).$ Since $h$ is a nullhomotopic map into $G$ we may find a path $h_t, \  0\leq t \leq 1 $ such that:

-  $h_t$ is constant near $0$ and $1$ 

-  $h_0\equiv 1\in G$ and $h_1=h.$

Next we use the map $h_t$ above to defined two gauge transformations on $X$ as follows. Recall that we denote by $T\times [0,1]$ with the coordinate $(e^{2\pi ix},e^{2\pi iy},r)$ a collar neighbourhood of $\partial X$ in $X$ such that $T\times\{1\}=\partial X.$  We define $g_x:T\times [0,1]\rightarrow G$ and $g_y:T\times [0,1]\rightarrow G$ by $g_x (e^{2\pi ix},e^{2\pi iy},r)=h_r(e^{2\pi ix})$ and 
$g_y (e^{2\pi ix},e^{2\pi iy},r)=h_r(e^{2\pi iy}).$
Note that we can extend $g_x$ and $g_y$ to be $1\in G$ outside the collar neighbourhood. 

It is not hard to check that $g_x \cdot A$ and $g_y \cdot A$ are also in normal form since  near $\partial X$ we have:
   
$g_x\cdot A=\begin{pmatrix} 
 2\pi i(\alpha_1+1)dx+& 0& 0\\
2\pi i\beta_1dy&&\\
&&\\
0& 2\pi i(\alpha_2-1)dx+& 0 \\
&2\pi i\beta_2dy&\\ 
&&\\
0& 0 & 2\pi i(\alpha_3dx+\beta_3dy)
\end{pmatrix}$

 $g_y\cdot A=\begin{pmatrix} 
 2\pi i\alpha_1dx+& 0& 0\\
 2\pi i(\beta_1+1)dy&&\\
 &&\\
0& 2\pi i\alpha_2dx+& 0 \\ 
&2\pi i(\beta_2-1)dy& \\
&&\\
0& 0 & 2\pi i(\alpha_3dx+ \beta_3dy)
\end{pmatrix}.$

We now state the following theorem:
\begin{theorem} Suppose that $A$ and $B$ are two gauge equivalent normal form flat $G$ connections on $X$ such that 

$A=\begin{pmatrix} 
 2\pi i(\alpha_1dx+\beta_1dy)& 0& 0\\
0& 2\pi i(\alpha_2dx+\beta_2dy)& 0 \\ 
0& 0 & 2\pi i(\alpha_3dx+ \beta_3dy)
\end{pmatrix}$ and 

$B=\begin{pmatrix} 
 2\pi i(\alpha_1+m)dx+& 0& 0\\
 2\pi i(\beta_1+n)dy&&\\
 &&\\
0& 2\pi i(\alpha_2-m)dx+& 0 \\
&2\pi i(\beta_2-n)dy&\\ 
&&\\
0& 0 & 2\pi i(\alpha_3dx+ \beta_3dy)
\end{pmatrix}$

near the boundary $\partial X,$ for some integers $m$ and $n,$ then:
$$cs(B)-cs(A)\equiv m(\beta_1-\beta_2)/2 - n(\alpha_1-\alpha_2)/2 \mod \Z.$$
\end{theorem}
\begin{proof} By Lemma 4.3, if $g$ is a gauge transformation such that $g\cdot A\equiv B$ near $\partial X$ then $cs(g\cdot A)\equiv cs(B) \mod \Z.$ Therefore, it is enough to prove that 
$cs(g_x\cdot A)-cs(A)\equiv (\beta_1-\beta_2)/2\mod \Z$ and  $cs(g_y\cdot A)-cs(A)\equiv (\alpha_2-\alpha_1)/2\mod \Z.$

By Proposition 1.27(e) of  \cite{freed} the difference between two Chern-Simons forms $CS(g_x\cdot A)-CS(A)$  equals to 
$$\frac{1}{8\pi^2}d(Tr(g_x^{-1}Ag_x\wedge g_x^{-1}dg_x)) -\frac{1}{24\pi^2}Tr(g_x^{-1}dg_x\wedge g_x^{-1}dg_x\wedge g_x^{-1}dg_x).$$
   It follows from the definition of $g_x$ that $\frac{\partial g_x}{\partial y}=0,$ so the last term in this formula vanishes.  By direct computation  we deduce that:
   
   $cs(g_x\cdot A)-cs(A)\equiv \frac{1}{8\pi^2}\int_T Tr(g_x^{-1}Ag_x\wedge g_x^{-1}dg_x) = \frac{1}{2}\int_T(\beta_1-\beta_2)dxdy$
   
   $\hskip3cm \equiv (\beta_1-\beta_2)/2.$

By using a similar argument, we can show that the formula for $g_y$ also holds. Thus the Theorem follows.
\end{proof}

\section{Applications}
In this section, we will apply our main theorem to find the Chern-Simons invariants of representations of the Seifert fibered homology sphere. We also give an explicit example, where we use the Chern-Simons invariant to deduce result about the Burns-Epstein invariant.

 Let $\Sigma=\Sigma(a_1,\ldots,a_n)$ be a Seifert fibered homology sphere, where $a_i>1$ are pairwise relatively prime integers. We put $a:=a_1\cdots a_n.$ We will denote by $\RR^*(\Sigma)$ the space of irreducible representations from $\fg(\Sigma)$ to  $G.$

The fundamental group of $\Sigma$ is given by :
$$\fg(\Sigma) = \{x_1,\ldots, x_n,h\ |h\ \text{is central,}\ \ x_1^{a_1}h^{b_1}= \cdots = x_n^{a_n}h^{b_n}=x_1\ldots x_n=1\}$$   
  where the $b_i$ are chosen with $\sum_1^n \frac{b_i}{a_i}=\frac{1}{a}.$ 
  
Let $\rho$ be an element of  $\RR^*(\Sigma).$ Since $h$ is in the center of $\fg(\Sigma),$ $\rho(h)$ is in the center of $G.$ As $x_i^{a_i}h^{b_i}=1, $ $\rho(x_i)$ is an elliptic element for all $i.$   Suppose that the representation  $\rho$ is given by 
$$\rho(h)=(diag(e^{2\pi ip_0},e^{2\pi iq_0},e^{2\pi ir_0}),2\pi(p_0+q_0+r_0),2\pi r_0) \ \text{and}$$ 
$$\rho(x_i) \sim (diag(e^{2\pi ip_i},e^{2\pi iq_i},e^{2\pi ir_i}),2\pi(p_i+q_i+r_i),2\pi r_i).$$ Here we use $\sim$ to denote the conjugacy relation in $G.$

As $\rho(h)$ is central, we have $p_0\equiv q_0\equiv r_0 \mod \Z.$ 
Since $\rho$ must preserve the relation $x_i^{a_i}h^{b_i}=1, \ i=1\ldots n,$ we deduce that :
 
(1) \hskip 2cm  $a_ip_i+b_ip_0=s_i,$ $(a_iq_i+b_iq_0)=-s_i$ are integers and 

(2) \hskip 2cm  $a_ir_i+b_ir_0=0.$ 

\begin{theorem} \label{sf}
The Chern-Simons invariant of $\rho$ is $$cs(\rho)\equiv   \frac{1}{2}a((\sum_{i=1}^np_i)^2+(\sum_{i=1}^nq_i)^2+(\sum_{i=1}^nr_i)^2).$$
\end{theorem}  
\begin{proof}
We write $\Sigma = X \cup (-S),$ where $X$ is the complement of the $n^{th}$-exceptional fiber and $S$ is the solid torus neighborhood of the $n^{th}$-exceptional fiber. We then find paths of representations on $X$ and $S$ which connect $\rho$ to the trivial representation and apply formula in Theorem 5.1 to compute the Chern-Simons invariants on each $X$ and $S.$

\noindent \textit {Step 1. Computation on $X.$}

Note that $\fg(X)$ is obtained from $\fg(\Sigma)$ by eliminating the relation $x_n^{a_n}h^{b_n}$ and $X$ has a natural meridian and longitude $\mu=x_n^{a_n}h^{b_n}$ and $\lambda = x_n^{-a_1\ldots a_{n-1}}h^c$ where $c=a_1\ldots a_{n-1}\sum_1^{n-1} \frac{b_i}{a_i}.$

 After conjugation, we may assume that 
$$\rho(x_n)=(diag(e^{2\pi ip_n}, e^{2\pi iq_n}, e^{2\pi ir_n}),2\pi(p_n+q_n+r_n),2\pi r_n).$$
Since each conjugacy class in $G$ is connected, we can deform $ \rho|_X(x_i)$ within its conjugacy class to the diagonal form for $i=1,\dots,n-1.$ 
This means that we can find a path of representations 
$$\rho_t:\fg(X)\longrightarrow G, 0\leq t\leq 1,$$
 such that $ \rho_0= \rho|_X,$ $ \rho_t(h)= \rho(h)$ for all $t$  and
 $$ \rho_1(x_i)=(diag(e^{2\pi ip_i},e^{2\pi iq_i},e^{2\pi ir_i}),2\pi(p_i+q_i+r_i),2\pi r_i), i=1,\dots,n-1.$$
 For this path of representation we have:
$$ \rho_t(x_n)=(diag(e^{2\pi ip(t)}, e^{2\pi iq(t)}, e^{2\pi ir(t)}),2\pi(p(t)+q(t)+r(t)),2\pi r(t)),$$ 
where $p(0)=p_n, q(0)=q_n, r(0)=r_n$ and

\noindent $p(1)=-\sum_1^{n-1}p_i, q(1)=-\sum_1^{n-1}q_i, r(1)=-\sum_1^{n-1}r_i.$

Therefore, we get:
$$\rho_t(\mu)=(M(t), 2\pi a_n(p(t)+q(t)+r(t))+2\pi b_n(p_0+q_0+r_0),2\pi a_nr(t)+2\pi b_nr_0),$$
$$\text{where}\ M(t)=diag(e^{2\pi i(a_np(t)+b_np_0)}, e^{2\pi i(a_nq(t)+b_nq_0)}, e^{2\pi i(a_nr(t)+b_nr_0)}), \ \text{and}$$

\noindent $\rho_t(\lambda)=(N(t), -2\pi a_1\cdots a_{n-1}(p(t)+q(t)+r(t))+2\pi c(p_0+q_0+r_0),$
$$ -2\pi a_1\cdots a_{n-1}r(t)+2\pi cr_0),\ \text{where}\hskip3.5cm $$
$$N(t)=diag(e^{2\pi i(-a_1\cdots a_{n-1}p(t)+cp_0)}, e^{2\pi i(-a_1\cdots a_{n-1}q(t)+cq_0)}, e^{2\pi i(-a_1\cdots a_{n-1}r(t)+cr_0)}).$$
So by Proposition 4.2 we get a path of normal form $\u21-$valued flat connections on $X$ which is given near $\partial X$ by  :
$$A_t=diag(2\pi i((a_np(t)+b_np_0)dx+(-a_1\cdots a_{n-1}p(t)+cp_0)dy), 2\pi i((a_nq(t)+b_nq_0)dx+$$
$$(-a_1\cdots a_{n-1}q(t)+cq_0)dy), 2\pi i((a_nr(t)+b_nr_0)dx+(-a_1\cdots a_{n-1}r(t)+cr_0)dy)).$$
We now use Theorem 5.1 to compute the difference of Chern-Simons invariants:

$cs(A_1)-cs(A_0)\equiv -\frac{1}{2}\int_0^1(a_nc+a_1\cdots a_{n-1}b_n)(p_0\dot p(t)+q_0\dot q(t))+ r_0\dot r(t))dt$

\hskip2.6cm $=-\frac{1}{2}(a_nc+a_1\cdots a_{n-1}b_n)(p_0p(t)+q_0q(t)+r_0r(t))|_0^1.$

So we arrive at:
$$(*) \qquad  cs(A_0)\equiv cs(A_1) -\frac{1}{2}(p_0\sum_{i=1}^np_i+q_0\sum_{i=1}^nq_i+r_0\sum_{i=1}^nr_i).$$
Now, by using relations  (1) and (2), we can write   

\noindent $A_1=diag(2\pi i(-a_n\sum_{i=1}^{n-1}p_i+b_np_0)dx+2\pi i(a_1\cdots a_{n-1}\sum_{i=1}^{n-1}\frac{s_i}{a_i})dy, 2\pi i(-a_n\sum_{i=1}^{n-1}q_i$

\hskip0.8cm $+b_nq_0)dx- 2\pi i(a_1\cdots a_{n-1}\sum_{i=1}^{n-1}\frac{s_i}{a_i})dy, 2\pi i(a_n\sum_{i=1}^{n-1}r_i+b_nr_0)dx).$

Since $a_1\cdots a_{n-1}\sum_{i=1}^{n-1}\frac{s_i}{a_i}$ is integer, by using Theorem 5.2, we find that:

\noindent $(**) \quad cs(A_1)\equiv cs(A_1')-\frac{1}{2}(a_1\cdots a_{n-1}\sum_{i=1}^{n-1}\frac{s_i}{a_i})(-a_n\sum_{i=1}^{n-1}p_i+b_np_0+$

$\hskip3cm a_n\sum_{i=1}^{n-1}q_i-b_nq_0).$

\noindent Where $A_1'=diag(2\pi i(-a_n\sum_{i=1}^{n-1}p_i+b_np_0)dx, 2\pi i(-a_n\sum_{i=1}^{n-1}q_i+b_nq_0)dx,$

\hskip1.8cm $2\pi i(-a_n\sum_{i=1}^{n-1}r_i+b_nr_0)dx),$ near the boundary $\partial X.$

As the holonomy representation $\rho_1'$ of the $A_1'$ is abelian, it factors through $\fg(X)\rightarrow H_1(X)\equiv \Z.$ We find that $\rho_1'(\lambda)=1\in G$ and

\noindent $\rho_1'(\mu)=(diag(e^{2\pi i(-a_n\sum_{i=1}^{n-1}p_i+b_np_0)}, e^{2\pi i(-a_n\sum_{i=1}^{n-1}q_i+b_nq_0)}, e^{2\pi i(-a_n\sum_{i=1}^{n-1}r_i+b_nr_0)}), $

\hskip0.5cm $-2\pi a_n \sum_{i=1}^{n-1} (p_i+q_i+r_i)+2\pi b_n(p_0+q_0+r_0), -2\pi a_n\sum_{i=1}^{n-1}r_i+b_nr_0).$

So we can deform $\rho_1'$ to the trivial representation and get a path of connections $A'_t$ linking $A_1'$ to the trivial connection which obviously has zero Chern-Simons invariant. Applying Theorem 5.1 to the path:

$A'_t =diag(2t\pi i(-a_n\sum_{i=1}^{n-1}p_i+b_np_0)dx, 2t\pi i(-a_n\sum_{i=1}^{n-1}q_i+b_nq_0)dx,$

\hskip1.8cm $2t\pi i(-a_n\sum_{i=1}^{n-1}r_i+b_nr_0)dx),\ 0\leq t\leq 1,$

we conclude that $cs(A'_1)\equiv 0.$ Now combine this with (*) and (**), we get:

$cs(A_0)\equiv -\frac{1}{2}a(\sum_{i=1}^{n-1}\frac{s_i}{a_i})(-\sum_{i=1}^{n-1}p_i+\frac{b_n}{a_n}p_0+\sum_{i=1}^{n-1}q_i-\frac{b_n}{a_n}q_0) -$

$\hskip3.5cm 
\frac{1}{2}(p_0\sum_{i=1}^np_i+q_0\sum_{i=1}^nq_i+r_0\sum_{i=1}^nr_i).$

Where $A_0$ is the flat connection form corresponding the representation $\rho|_X.$ Now using (1) we can rewrite this as:

$(***) \qquad  cs(A_0)\equiv -\frac{1}{2}a(\sum_{i=1}^{n-1}\frac{s_i}{a_i})(-\sum_{i=1}^{n}p_i+\sum_{i=1}^{n}q_i+2\frac{s_n}{a_n}) -$

$\hskip3.8cm 
\frac{1}{2}(p_0\sum_{i=1}^np_i+q_0\sum_{i=1}^nq_i+r_0\sum_{i=1}^nr_i).$

\noindent \textit {Step 2. Computation on the solid torus}

We will denote the connection form corresponding to $\rho|_S$ by $B_0.$ Near the boundary $\partial X$ then $B_0$ coincides with $A_0$ and given by:

\noindent $B_0=diag(2\pi i((a_np_n+b_np_0)dx+(-a_1\cdots a_{n-1}p_n+cp_0)dy)), 2\pi i((a_nq_n+$

\hskip0.3cm $b_nq_0)dx+ (-a_1\cdots a_{n-1}q_n+cq_0)dy)), 2\pi i(-a_1\cdots a_{n-1}r_n+cr_0)dy).$

By  (1), the numbers $a_np_n+b_np_0=s_n$ and $a_nq_n+b_nq_0=-s_n$ are integers. So apply Theorem 5.2 we find that 

$cs(B_0)\equiv cs(B_1)+\frac{1}{2}s_n( -a_1\cdots a_{n-1}p_n+cp_0+a_1\cdots a_{n-1}q_n-cq_0),$

where

 $B_1 = diag(2\pi i(-a_1\cdots a_{n-1}p_n+cp_0)dy, 2\pi i (-a_1\cdots a_{n-1}q_n+$

\hskip2cm $cq_0)dy, 2\pi i(-a_1\cdots a_{n-1}r_n+cr_0)dy)$ near the boundary $\partial X.$

Note that  $B_1$ is a connection form corresponding to an abelian representation and it involves only the $dy$ terms. Carrying out a similar computation as we did for the connection $A_1'$ in the previous step, we conclude that $cs(B_1)\equiv 0.$ Therefore we can write
$$cs(B_0)\equiv \frac{1}{2}a\frac{s_n}{a_n}( -p_n+p_0\sum_{i=1}^{n-1}\frac{b_i}{a_i}+q_n-q_0\sum_{i=1}^{n-1}\frac{b_i}{a_i}).$$
Moreover, from (1), we deduce that $$p_0\sum_{i=1}^{n-1}\frac{b_i}{a_i}=\sum_{i=1}^{n-1}\frac{s_i}{a_i}-\sum_{i=1}^{n-1}p_i$$ and that 
$$q_0\sum_{i=1}^{n-1}\frac{b_i}{a_i}=-\sum_{i=1}^{n-1}\frac{s_i}{a_i} -\sum_{i=1}^{n-1}q_i.$$ So we arrive at:
$$(****)\qquad cs(B_0)\equiv \frac{1}{2}a\frac{s_n}{a_n}( -\sum_{i=1}^np_i+\sum_{i=1}^nq_i+2\sum_{i=1}^{n-1}\frac{s_i}{a_i}).$$
Now as  $\Sigma = X \cup (-S),$ we see that $cs(\rho)=cs(A_0)-cs(B_0).$ Note that the term $a\frac{s_n}{a_n}\sum_{i=1}^{n-1}\frac{s_i}{a_i}$ appearing in  (***) and (****) is an integer and therefore can be ignored. What we find is 

$cs(\rho)\equiv \frac{1}{2}a(\sum_{i=1}^np_i\sum_{i=1}^n\frac{s_i}{a_i}-\sum_{i=1}^nq_i\sum_{i=1}^n\frac{s_i}{a_i}) $

\hskip3cm $-\frac{1}{2}(p_0\sum_{i=1}^np_i+q_0\sum_{i=1}^nq_i+r_0\sum_{i=1}^nr_i).$

It follows from (1) and (2) that: 

$\sum_{i=1}^n\frac{s_i}{a_i}=\sum_{i=1}^np_i+\frac{p_0}{a}=-\sum_{i=1}^nq_i-\frac{q_0}{a}$ and 
$r_0=-a\sum_{i=1}^nr_i.$ 

So, finally, we arrive at the needed formula
$$cs(\rho)\equiv   \frac{1}{2}a((\sum_{i=1}^np_i)^2+(\sum_{i=1}^nq_i)^2+(\sum_{i=1}^nr_i)^2).$$ 
\end{proof}

We can use the above theorem to find \textit {all the possible values} of the Chern-Simons invariants of representations if we  know the representation space $\RR^*(\Sigma).$ In the general case, this space is still hard to describe in details. Fortunately, for the case of Seifert fibered homology sphere with three singular fibers, the representation space has been studied in \cite{k2} and so we are able to find all the the possible values of the Chern-Simons invariants.

 As an illustration,  we present an example of the homology sphere $\Sigma(2,3,11).$ Its fundamental group has the following presentation.
{\small
$$\fg(\Sigma(2,3,11))= \langle
x_1,x_2,x_3,h| \ h\ \text{central},\ x_1^2h^{-1}=x_2^3h=x_3^{11}h^2=x_1x_2x_3=1
\rangle.$$}
By the computation in  \cite{k2}, we know that   $\Sigma(2,3,11)$ has five distinct irreducible representations into $\PU.$ By homological reason, each $\PU$ representation has a unique lift to a representation into the universal covering group $G.$ By  further computations, we obtain the following list of representations into $G.$ 

\noindent 1) $\rho(x_1) \sim (diag(1,-1,-1), 0,-\pi), \rho(x_2) \sim  (diag(1,e^{4\pi i/3},e^{2\pi i/3}),0, \frac{2\pi}{3}),$ 

$ \rho(x_3) \sim  (diag(e^{12\pi i/11},e^{6\pi i/11},e^{4\pi i/11}),0, \frac{4\pi}{11}),\quad \rho(h)\sim (I, 0,-2\pi).$

\noindent 2) $\rho(x_1) \sim (diag(1,-1,-1),0,\pi),  \rho(x_2) \sim  (diag(e^{2\pi i/3}, 1,e^{4\pi i/3}),0, -\frac{2\pi}{3})),$ 

$ \rho(x_3) \sim  (diag(e^{-12\pi i/11},e^{-6\pi i/11},e^{-4\pi i/11}),0,-\frac{4\pi}{11}),\quad  \rho(h)\sim (I, 0,2\pi).$
 
\noindent 3) $\rho(x_1) \sim (diag(-1,-1, 1), 0, 2\pi), \rho(x_2) \sim  (diag(1,e^{4\pi i/3},e^{2\pi i/3}),0,- \frac{4\pi}{3})),$ 

$ \rho(x_3) \sim  (diag(e^{-4\pi i/11},e^{-10\pi i/11},e^{-8\pi i/11}),0,-\frac{8\pi}{11}),\quad  \rho(h)\sim (I, 0,4\pi).$

\noindent 4) $\rho(x_1) \sim (diag(-1,-1, 1), 0, -2\pi), \rho(x_2) \sim  (diag(e^{2\pi i/3},1,e^{4\pi i/3}),0,\frac{4\pi}{3})),$ 

$ \rho(x_3) \sim  (diag(e^{4\pi i/11},e^{10\pi i/11},e^{8\pi i/11}),0, \frac{8\pi}{11}), \quad \rho(h)\sim (I, 0,-4\pi).$

\noindent 5) $\rho(x_1) \sim (diag(-1,-1, 1), 0, 0), \rho(x_2) \sim  (diag(e^{2\pi i/3},e^{4\pi i/3},1),0,0),$ 

$ \rho(x_3) \sim  (diag(e^{-2\pi i/11},e^{2\pi i/11},1),0, 0),\quad \rho(h)\sim (I, 0, 0).$

Using Theorem \ref{sf}, we can find the values of  the Chern-Simons invariants of $\Sigma(2,3,11)$  as below. 

\begin{center}
 \begin{tabular}{cccccc}
 Case \qquad &\qquad 1 &\qquad 2 &\qquad 3 &\qquad 4 &\qquad 5 \\
 $cs(\rho) \mod \Z $ &\qquad $\frac{13}{66}$ &\qquad $\frac{13}{66} $ &\qquad $ \frac{7}{66}$ &\qquad $\frac{7}{66}$  &\qquad $\frac{25}{66}$
\end{tabular}
\end{center}

 Therefore, we can deduce that: \textit{the Burn-Epstein invariant $(\mod \Z)$ of any spherical CR structure on $\Sigma(2,3,11)$ with irreducible holonomy representation is one of the values above. }
\begin{remark}
 - Little is known about the problem of classification of spherical CR structures on 3-manifolds.  On Seifert fibered  manifolds,  the only work done so far is the classification of the $\S^1$-invariant CR structure by Kamishima and Tsuboi \cite{kt}. We do not know any example of spherical CR structure whose holonomy is the representations given in five cases above.
 
- Recently,  Biquard and Herzlich \cite{bi1} introduce an invariant $\nu$ for strictly pseudoconvex $3$-dimensional CR manifold $M$ and show that their invariant agrees with three times the Burn-Epstein invariant up to a constant. Furthermore, by relating the $\nu$ invariant to a kind of eta invariant, Biquard, Herzlich and Rumin \cite{bi2} are able to give an explicit formula for the $\nu$ invariant of the transverse $\S^1$-invariant CR structure on a Seifert fibered manifold. It would be interesting to work out the relation between the $\nu$ invariant, modulo an integer,  and the metric Chern-Simons invariant. 
 \end{remark} 


\begin{thebibliography}{Rubinsteinetal}



\bibitem{bi1} Biquard, O., Herzlich, M.,  \textit {A Burns-Epstein invariant for ACHE 4-manifolds,}
Duke Math. J. 126 (2005), 53-100.

\bibitem{bi2} Biquard, O.,  Herzlich, M., and  Rumin, M.,  \textit{Diabatic limit, eta invariants, and Cauchy-Riemann manifolds of dimension 3,}
Ann. Scient. Ec. Norm. Sup. 40 (2007) 589-631.

\bibitem{be1}  Burns D., Epstein C., \textit{A global invariant for three-dimensional CR-manifolds,} Invent. Math. \textbf{92} (1988), no. 2, 333--348. 

\bibitem{be2}  Burns D., Epstein C., \textit{Characteristic numbers of bounded domains,} Acta Math. \textbf{164} (1990), no. 1-2, 29--71. 

\bibitem{chen}Chen, S.S.; Greenberg, L, \textit{Hyperbolic spaces,} Contributions to Analysis, Collect. of Papers dedicated to Lipman Bers, 49- 87 (1974). 

\bibitem{cheng} Cheng, J.-H.,  Lee, J. M., \textit{The Burns-Epstein invariant and deformation of CR structures,} Duke Math. J. \textbf{60} (1990), 221-254 

\bibitem{cs} Chern, S.S., Simons, J. \textit{characteristic forms and geometric invariants,} Ann. of Math. (2) \textbf{99} (1974), 48-69.

\bibitem{falbel}  Falbel, Elisha, Gusevskii, Nikolay,  \textit{Spherical CR-manifolds of dimension 3,} Bol. Soc. Brasil. Mat. (N.S.) \textbf{25} (1994), no. 1, 31--56.

\bibitem{freed}
 Freed, Daniel S., \textit{Classical Chern-Simons theory I,}Adv. Math. \textbf{113} (1995), no. 2, 237--303. 
 

\bibitem{hansen}
Hansen, M.L., \textit{Week amenability of the universal covering group of ${\rm SU}(1,n)$,} Math. Ann. \textbf {288} (1990), no. 3, 445-472.

\bibitem{j} Jacobowitz, Howard, \textit{An introduction to CR structures.} Mathematical Surveys and Monographs, 32. American Mathematical Society, Providence, RI, 1990.

\bibitem{kt} Kamishima, Y.,  Tsuboi, T., \textit{CR-structures on Seifert manifolds,} Invent. Math. \textbf {104} (1991) 149-163.

\bibitem{k1} Khoi, Vu The, \textit{ A cut-and-paste method for computing the Seifert volumes}, Math. Ann. \textbf {326} (2003), no. 4, 759-801. 

\bibitem {k2}Khoi, Vu The, \textit {On the $\SU$ representation space of the Brieskorn homology spheres}, J. Math. Sci. Univ. Tokyo
Vol. 14 (2007), No. 4, 499--510.


\bibitem{kk1} Kirk, P., Klassen, E., \textit{Chern-Simons invariants of 3-manifolds and representation spaces of knot groups,} Math. Ann. \textbf {287} (1990), no. 2, 343-367.

\bibitem{kk2}Kirk,  P., Klassen, E., \textit{Chern-Simons invariants of $3$-manifolds
decomposed along tori and the circle bundle over the representation space of $T\sp 2$,}
Comm. Math. Phys. \textbf{153} (1993), no. 3, 521--557.

\bibitem{nishi} Nishi, Haruko, \textit {$\SUn$ Chern-Simons invariants of Seifert fibered 3-manifolds,} Internat. J. Math., \textbf{9} (1998),  no. 3, 295-330. 

\end{thebibliography}
\end{document}